\documentclass[12pt, reqno]{amsart}
\allowdisplaybreaks[1]
\usepackage{amsmath}
\usepackage{amssymb}
\usepackage{amsfonts}
\usepackage{verbatim}
\usepackage{amsthm}
\usepackage[usenames]{color}
\usepackage{hyperref}
\makeindex

 \newtheorem{theorem}{Theorem}[section]
 
 \newtheorem{corollary}[theorem]{Corollary}
  \newtheorem{conjecture}[theorem]{Conjecture}
 \newtheorem{lemma}[theorem]{Lemma}
 \newtheorem{proposition}[theorem]{Proposition}

\theoremstyle{definition}

\theoremstyle{remark}
\newtheorem{rem}[theorem]{Remark}
\newtheorem{fact*}{Fact}
\DeclareMathOperator{\vecc}{vec}

\newcommand\dd{\mathrm d}

\DeclareMathOperator{\Tr}{Tr}

\newcommand{\cc}[1]{\overline{#1}}

\newcommand{\til}{\raise.17ex\hbox{$\scriptstyle\mathtt{\sim}$}}

\newcommand\beq{\begin{equation}}

\newcommand\eeq{\end{equation}}

\newcommand\black{\color{black}}

\newcommand{\bbm}{\left[ \begin{matrix}}
\newcommand{\ebm}{\end{matrix} \right]}
\newcommand{\bpm}{\left( \begin{matrix}}
\newcommand{\epm}{\end{matrix} \right)}
\numberwithin{equation}{section}

\newcommand{\tensor}[2]{\text{ }{\begin{smallmatrix} #1 \\ \otimes\\ #2\end{smallmatrix}}\text{  }}
\newcommand{\flattensor}[2]{#1 \otimes #2}

\newlength{\Mheight}
\newlength{\cwidth}

\newcommand{\dfn}[1]{{\bf #1}\index{#1}}

\title[The outer spectral radius]{The outer spectral radius and dynamics of completely positive maps}
\author{
J. E. Pascoe 
}
\date{\today}

\begin{document}

\begin{abstract}
We examine a special case of an approximation of the joint spectral radius given by Blondel and Nesterov,
which we call the outer spectral radius.
The outer spectral radius is given by the square root of
the ordinary spectral radius of the $n^2$ by $n^2$ matrix $\sum \flattensor{\overline{X_i}}{X_i}.$
We give an analogue of the spectral radius formula for the outer spectral radius which can be used to quickly obtain
the error bounds in methods based on the work of Blondel and Nesterov.
The outer spectral radius is used to analyze the iterates of a completely positive map, including the special case of quantum channels.
The average of the iterates of a completely positive map approach to a completely positive map where the Kraus operators span an ideal in the algebra
generated by the Kraus operators of the original completely positive map. We also give an elementary treatment of Popescu's theorems on similarity to row contractions in the matrix case,
describe connections to the Parrilo-Jadbabaie relaxation, and give a detailed analysis of the maximal spectrum of a completely positive map.
\end{abstract}
\maketitle

\tableofcontents

\section{Introduction}

The \dfn{spectral radius} of a matrix $X \in M_n(\mathbb{C}),$ denoted $\rho(X)$ is given by the maximum modulus of
the eigenvalues of $X.$ The spectral radius can also be computed via the Gelfand formula:
	$$\rho(X) = \lim_{k\rightarrow \infty} \|X^k\|^{1/k}.$$
The \dfn{joint spectral radius} of a tuple $(X_1,\ldots, X_d) \in M_n(\mathbb{C})^d$ is defined in terms of a Gelfand type formula to be:
$$\rho(X_1,\ldots, X_d) = \lim_{k\rightarrow \infty} \sup_{1 \leq i_1, \ldots, i_k \leq d} \| X_{i_1}\ldots X_{i_k} \|^{1/k}.$$
%\red Talk about how its been variously generalized already. \black
We define
the \dfn{outer spectral radius}, which is defined via the formula
	$$\hat{\rho}(X_1,\ldots , X_d) = \sqrt{\rho\left(\sum \flattensor{\overline{X_i}}{X_i}\right)}$$
where $\flattensor{A}{B}$ is just the usual Kronecker product and $\overline{A}$ is the complex conjugate of $A$.
The outer spectral radius is essentially a special case of the the Blondel-Nesterov approximation of the joint spectral
radius in \cite{blondnest05}. Moreover, the outer spectral radius relates to the quantum information theory literature as spectral radius of the completely positive map
$\phi(H)=\sum X_iHX_i^*$ \cite{Sanz, Hermes, Shitov, Raginsky, PopescuSimilarity, EKH}. Under this guise, although in an independent fashion, a Rota-Strang type theory has been developed by G. Popescu for the outer spectral radius \cite{PopescuSimilarity}. 
One goal will be to give a Gelfand type theorem, unify the various manifestations of the outer spectral radius and give an elementary treatment of the Popescu-Rota-Strang theory.
Moreover, we will give a detailed spectral analysis of $\sum \flattensor{\overline{X_i}}{X_i}.$ We note that, up to this point in time, it does not appear that the connection between
the Blondel-Nesterov relaxation of the joint spectral radius, the dynamics of completely positive maps and the Popescu theory of row contractions had been noticed.

Although the outer spectral radius is interesting in its own right as a natural relaxation of the joint spectral radius, the immediate question is:
what dynamics does it describe? It turns out there is a satisfying answer here as well- it describes the dynamics of the iterates of a completely positive map.
(This was the approach taken by Popescu in \cite{PopescuSimilarity}.)
Moreover, the same tools used to analyze the outer spectral radius can be used to do a more detailed analysis of $\sum \flattensor{\overline{X_i}}{X_i}$ itself.
The ``sinks" in terms of the dynamics here turn out to be ideals in the algebra generated by $X_1, \ldots, X_d.$
In the special case that the completely positive map is trace preserving, sometimes referred to  a quantum channel, the dynamics are particularly nice.

We can characterize the outer spectral radius $\hat{\rho}(X_1,\ldots, X_d)$ in terms of a Gelfand type formula.
\begin{theorem}\label{outerspectralradiusformula}
%If
%$(X_1,\ldots, X_d) \in M_n(\mathbb{C})^d,$ then,
Let $(X_1,\ldots, X_d) \in M_n(\mathbb{C})^d.$
$$\hat{\rho}(X_1,\ldots, X_d) = \lim_{k\rightarrow \infty} \sup_{
\sum |a_{i_1,\ldots,i_k}|^2 = 1}
\left\|\sum_{1 \leq i_1, \ldots, i_k \leq d}  a_{i_1,\ldots,i_k}X_{i_1}\ldots X_{i_k} \right\|^{1/k}.$$
\end{theorem}
Theorem \ref{outerspectralradiusformula} is proven Section \ref{outerspectralradiusformulaproof}.

Let $T$ be a square matrix.
We call the subset of the eigenvalues of $T$ with maximum modulus which have maximal degeneracy index the \dfn{maximal spectrum of $T$}. If the maximal degeneracy index is $1$ we say the maximal spectrum is 
    \dfn{nondegenerate}. (The degeneracy index of an eigenvalue is the maximum size of the Jordan blocks correspond to that eigenvalue.)
A canonical choice of positive $L$ such that $L - \sum X_iLX_i^*,$
obtained in Theorem \ref{mainresult}, our treatment of the Popescu's Rota-Strang theory,
has a special enough form to imply that the maximal eigenvalue of $\sum \flattensor{\overline{X_i}}{X_i}$ can be chosen to be non-negative and real-valued. (This is similar to the degenerate case of the classical
Perron-Frobenius theorem, and can be viewed as a degenerate case of the Quantum Perron-Frobenius theorem\cite{EKH}.)
That is, be have the following result.
\begin{theorem}[Degenerate Quantum Perron-Frobenius theorem]\label{maximalreal}
	Let $(X_1,\ldots, X_d) \in M_n(\mathbb{C})^d.$
	Let $T = \sum \flattensor{\overline{X_i}}{X_i}.$
	There is a non-negative real eigenvalue $\lambda$ of $T$ with degeneracy index $\eta$, such that any other eigenvalue $\lambda'$ of $T$ with degeneracy index $\eta'$
	has the property that either
    \begin{enumerate}    
        \item $\lambda= |\lambda'|$ and $\eta\geq\eta',$
        \item $\lambda > |\lambda'|.$
    \end{enumerate}
    That is, the maximal spectrum has a real nonnegative element.
\end{theorem}
Theorem \ref{maximalreal} is proven in Section \ref{maxsect}.
We give some further comments on the structure of the spectrum of $T$ in \ref{ParilloSection}, which explains some observations by Blondel-Nesterov and the relationship with the Parrilo-Jadbabaie relaxation of the
joint spectral radius.

\subsection{Dynamics of quantum channels and other completely positive maps}
The classical joint spectral radius can be seen as descibing the dynamics of switched linear systems \cite{DaubechiesLagarias, BlondelBirther}. In parallel, we can use the outer spectral radius theory to understand the dynamics of completely positive maps, and more specifically the dynamics of quantum channels, for which there has been some recent interest \cite{Sanz, Hermes, Shitov, Raginsky, PopescuSimilarity, EKH}.
That is we want to understand the iterates of a map $\phi:M_n(\mathbb{C}) \rightarrow M_n(\mathbb{C})$ of the form
    $$\phi(H)= \sum A_iHA_i^*$$
by the Choi-Kraus theorem \cite{cho72,kraus71}, the quantum channels are those completely positive maps such that additionally $\sum A_i^*A_i = I.$ These are the \dfn{trace preserving} completely positive maps. Maps satisfying
$\sum A_iA_i^* = I$ are called \dfn{unital.}
Given $T$ of the form $\sum \flattensor{\overline{A_i}}{A_i},$ we can write the map $\phi_T(H) = \sum A_iHA_i^*.$ (Moreover by the Choi-Kraus characterization, this is all of them.)
Note that $$\phi_{T_1T_2} = \phi_{T_1}\circ\phi_{T_2}.$$
Therefore, the dynamics of the map $\phi_T$ are essentially those of $T^n.$ Therefore, we can apply the typical idea in dynamics and study the part of $T$ corresponding to the maximal spectrum.

\subsubsection{The general case}
Let $T = \sum \flattensor{\overline{X_i}}{X_i}.$
Additionally, assume that $\hat{\rho}(X_1,\ldots , X_d)=1.$
Let $m_T$ denote the cardinality of the maximal  spectrum of $T.$
We define $\Lambda$ to be the closure of the subgroup of the torus $\mathbb{T}^{m_T}$ generated by the point $\tau=(\lambda_1, \ldots, \lambda_{m_T})$
where the $\lambda_i$ are the elements of the maximal spectrum.
As a topological space and a group, $\Lambda \cong \mathbb{T}^{c_T} \times H$ where $H$ is some finite abelian group and $c_T \in \mathbb{N}.$
There is natural bijective map $T_\lambda$ from $\Lambda$ to the set of limit points of the sequence $\frac{T^n}{\|T^n\|},$ since along any sequence such that
$\tau^n$ converges, the $\frac{T^n}{\|T^n\|}$ converges. (One can see this by considering the Jordan canonical form.)
Moreover, note that either $T_\lambda^2=T_{\lambda^2}$ or $T_\lambda^2 =0,$ depending on whether or not the maximal spectrum is nondegenerate.
Now by calculation, $$TT_\lambda=T_\lambda T = T_{\tau\lambda}.$$
The set of $T_\lambda$ describe the \dfn{asymptotic dynamics} of $T^n.$
Now, $$\hat{T} = \int_{\Lambda} T_\lambda \dd \lambda,$$ where the integral is taken with respect to normalized Haar measure on $\Lambda.$
Alternatively, we could take the elementary and equivalent definition $$\hat{T}= \lim_{N\rightarrow \infty}\frac{1}{N}\sum^N_{n=1} \frac{T^n}{\|T^n\|},$$ as for large $n,$ the quantity 
$\frac{T^n}{\|T^n\|}$ behaves like a random $T_{\lambda}.$
The matrix $\hat{T}$ essentially describes the average asymptotic dynamics. Note that either $\hat{T}^2 = \hat{T}$ or $\hat{T}^2 = 0.$
\begin{theorem}\label{dyntheorem}
    Let $T = \sum \flattensor{\overline{X_i}}{X_i}.$
    Additionally, assume that $\hat{\rho}(X_1,\ldots , X_d)=1.$
    \begin{enumerate}
    \item
    There is an $a \in \mathbb{N}$ such that for every $\lambda \in \Lambda$
    there are $A_{\lambda,1}\ldots A_{\lambda,a}\in M_n(\mathbb{C})$ such that
    $T_\lambda = \sum^a_{j=1} \flattensor{\overline{A_{\lambda,j}}}{A_{\lambda,j}}.$
    \item
     There are $B_{k}\in M_n(\mathbb{C})$ whose span is a nonzero ideal in the algebra generated by $X_1,\ldots, X_d$ such that
    $\hat{T} = \sum_k \flattensor{\overline{B_{k}}}{B_{k}}.$ Moreover, the $A_{\lambda,j}$ are contained in the span of the $B_{k}$'s.
    \end{enumerate}
\end{theorem}
Theorem \ref{dyntheorem} is proved as Proposition \ref{dynprop}

We can interpret the above theorem as saying that a sink in this framework is somehow corresponded to an ideal $I$ spanned by the $B_k$'s.
Moreover, this ideal is weakly graded into the subspaces $I_\lambda$ spanned by $A_{\lambda,j}$ in the sense that the sum of all the $I_\lambda$ is $I,$
and $\Lambda$ naturally acts on the $I_\lambda.$ Note that 
$$TT_\lambda=T_{\tau\lambda}=\sum^d_{i=1} \sum^a_{j=1} \flattensor{\overline{X_iA_{\lambda,j}}}{X_iA_{\lambda,j}}= \sum^a_{j=1} \flattensor{\overline{A_{\tau\lambda,j}}}{A_{\tau\lambda,j}}.$$
 That is, the movement of $I_\lambda$ as $\lambda$ ranges through $\Lambda$ somehow should be thought of as an asymptotic orbit.
Note these dynamics describe the limiting coefficients of the channel $\phi_{T^n/\|T^n\|}$ and not the range of such a channel.
In the case where $T_\lambda^2 = 0,$ the ideal $I$ must satisfy that $I^2=0$.
Moreover, whenever the maximal spectrum contains a single element, as happens generically, the group $\Lambda$ has a single element, and therefore $\hat{T} = T_\lambda,$ which in turn should be thought of as saying the action has no circulation in the limit or is ergodic.

Whenever $X_1, \ldots, X_d$ generate the full algebra of $n$ by $n$ matrices, average asymptotic dynamics simplify. The condition that the coordinates generate the full algebra
essentially corresponds to the condition of having all positive entries in the classical Perron-Frobenius theorem. Under an irreducibility type assumptions various Quantum Perron-Frobenius theorems which establish
the existence of a simple real eigenvalue with maximum modulus have been obtained by Evans-Hoegh-Krohn\cite{EKH}, Schrader\cite{shred} and Lagro-Yang-Xiong\cite{Lagro2017}. 
We essentially gather some more detailed structure than the aforementioned works in the finite dimensional case by applying Theorem \ref{dyntheorem}.
\begin{theorem}[Quantum Perron-Frobenius Theorem] \label{fullmaximal}
     Let $T = \sum \flattensor{\overline{X_i}}{X_i}.$
    Additionally, assume that $\hat{\rho}(X_1,\ldots , X_d)=1.$
    If $X_1, \ldots, X_d$ generate the full algebra of $n$ by $n$ matrices, then $\hat{T}$ has rank $1.$
    Moreover, in such a case, $\hat{T}=\sum_k \flattensor{\overline{B_{k}}}{B_{k}}$ where the $B_k$'s span $M_n(\mathbb{C}).$    
    That is, the ideal described in Theorem \ref{dyntheorem} is equal to the whole algebra, as the only nonzero ideal in $M_n(\mathbb{C})$ is  $M_n(\mathbb{C})$  itself.
\end{theorem}
Theorem \ref{fullmaximal} is proved as Proposition \ref{fullmaximalprop}. Corollary \ref{fullmaximalcor} establishes that whenever the $X_i$'s generate the full algebra, then it can be conjugated to a row co-isometry or column isometry, therefore the corresponding conjugated channels are unital and trace preserving respectively.

\subsubsection{Nondegenerate maximal spectrum and quantum channels}
If the maximal spectrum is nondegenerate, the situation simiplifies somewhat.
Firstly, $T_\lambda T_{\lambda'} = T_{\lambda\lambda'}.$ That is, the map taking $\lambda$ to $T_\lambda$ behaves essentially like a group homomorphism. (However, it is possible that
$T_\lambda$ may have a nonzero kernel.) We also note that if $T$ describes a quantum channel, then $T$ must have nondegenerate maximal spectrum, as otherwise, there would be a positive $H$ such that
$\|\phi_{T^n}(H)\|\rightarrow \infty.$ However, this cannot happen as $\phi_T$ is trace preserving and a completely positive map.

% Say that A_i \otime A_i gets sent to something else

% since
%$\hat{\rho}(X) = \rho\left(\tensor{\overline{X}}{X}\right)<1$ if and only if $\rho(X) <1.$
%Additionally, it seems likely that the
%problem is significantly more difficult in the case of operators acting on a general Hilbert space, since this
%is the case even when $d=1$ \cite{}.

\subsection{The outer spectral radius formula and the Blondel-Nesterov approximation}

We will now discuss how Theorem \ref{outerspectralradiusformula} can be used to exhibit some well known approximations of
the joint spectral radius which originated with Blondel and Nesterov in \cite{blondnest05}, and has been improved in
various ways in \cite{xu2011, par08}.

Using Theorem \ref{outerspectralradiusformula} we can relate the outer spectral radius to the joint spectral radius.
We leave the details of this calculation to the reader.
\begin{corollary}\label{outerspectralradiusformulacor}
Let $(X_1,\ldots, X_d) \in M_n(\mathbb{C})^d.$
%If
%$(X_1,\ldots, X_d) \in M_n(\mathbb{C})^d,$
%then
$$\frac{1}{\sqrt{d}}\hat{\rho}(X_1,\ldots, X_d) \leq \rho(X_1,\ldots, X_d) \leq \hat{\rho}(X_1,\ldots, X_d).$$
\end{corollary}

That is, the outer spectral radius itself serves as an approximation of the joint spectral radius. We note that the outer spectral radius is relatively easy to compute, (since
it is just the maximum modulus eigenvalue of an $n^2$ by $n^2$ matrix) whereas the 
joint spectral radius is somewhat difficult\cite{blondel2,blondel}, although there
are several ways to approximate the joint spectral radius using techniques from optimization\cite{blondnest05, par08,par14}.

Blondel and Nesterov\cite{blondnest05} showed that, for tuples of matrices with nonnegative entries,
	$$\rho(X_1,\ldots, X_d) = \lim_{k\rightarrow \infty} \rho\left(\sum X_i^{\otimes k}\right)^{1/k}.$$
Here $A^{\otimes n}$ denotes the $n$-th Kronecker power of $A,$ the Kronecker product of $A$ with itself $n$ times. Moreover, they obtained inequalities as in \ref{outerspectralradiusformulacor}.
In \cite{xu2011}, Xiao and Xu showed that, in general,
	$$\rho(X_1,\ldots, X_d) = \limsup_{k\rightarrow \infty} \rho\left(\sum X_i^{\otimes k}\right)^{1/k}.$$
We use the outer spectral radius to generalize the Blondel-Nesterov formula in a way that eliminates the supremum from the Xiao-Xu formula and
gives bounds on the error which are the same as the error bounds in the Blondel-Nesterov\cite{blondnest05} approximation.
That is, we give a family of asymptotically tight approximations which converge to the joint spectral radius in Corollary \ref{osrfccor}
for arbitrary matrices over $\mathbb{C}.$
Specifically,
	\begin{align*}
	\rho(X_1,\ldots, X_d) & = \lim_{k\rightarrow \infty} \hat{\rho}(X_1^{\otimes k},\ldots, X_d^{\otimes k})^{1/k}, \\
	& = \lim_{k\rightarrow \infty} \rho\left(\sum \overline{X_i^{\otimes k}}  \otimes X_i^{\otimes k}\right)^{1/2k}.
	\end{align*}

In fact, Xiao and Xu\cite{xu2011} noted that, even for real matrices,
$\lim_{k\rightarrow \infty} \rho\left(\sum X_i^{\otimes k}\right)^{1/k}$ may not exist. We can now, in light of our new
approximation, view this phenomenon as stemming from the fact that $k$ could be odd, since, over the reals, our formula
implies that $\lim_{k\rightarrow \infty} \rho\left(\sum X_i^{\otimes 2k}\right)^{1/2k}$ exists.

To obtain our formula, we observe that
$$\rho(X_1^{\otimes k},\ldots, X_d^{\otimes k}) = \rho(X_1,\ldots, X_d)^k.$$
So, we obtain a Blondel-Nesterov type formula for
the joint spectral radius as a consequence of
Corollary \ref{outerspectralradiusformulacor}.
	 \begin{corollary} \label{osrfccor}
	 Let $(X_1,\ldots, X_d) \in M_n(\mathbb{C})^d.$
	 	$$\frac{1}{\sqrt[2k]{d}}\hat{\rho}(X_1^{\otimes k},\ldots, X_d^{\otimes k})^{1/k} \leq \rho(X_1,\ldots, X_d) \leq \hat{\rho}(X_1^{\otimes k},\ldots, X_d^{\otimes k})^{1/k}.$$
	 	Namely,
		$$\rho(X_1,\ldots, X_d) = \lim_{k\rightarrow \infty} \hat{\rho}(X_1^{\otimes k},\ldots, X_d^{\otimes k})^{1/k}.$$
	\end{corollary}

\subsection{A Popescu-Rota-Strang theory of the outer spectral radius}
We can obtain
an analogue of a theorem on the joint spectral radius
of Rota and Strang\cite{rotastrang, rotamodel}
which states that:
	$$\rho(X_1,\ldots,X_d)= \inf_{N \in \mathfrak{N}} \sup_{i} \|X_i\|_N$$
where $\mathfrak{N}$ is the set of all consistent matrix norms on $M_n(\mathbb{C}).$
The outer spectral radius is the infimum over all ``two norms'' of the block matrix $\bbm X_1 & \ldots & X_d \ebm,$
which can be made formal as the infimum over all points simultaneously similar to $(X_1,\ldots,X_d)$ of the norm as a block matrix. 
A version of the Rota-Strang theory was developed by Popescu in \cite{PopescuSimilarity} for the row ball viewed the set of Kraus coefficients of iterable contractive completely positive maps which under translation
becomes a Rota-Strang theory of the outer spectral radius initiated by Blondel-Nesterov. (Apparently, that the two quantities were the same had not been realized.)
 One of our goals will be to give an elementary treatment of Popescu's Rota-Strang theory for the outer spectral radius on its own terms.
\begin{theorem}[Popescu \cite{PopescuSimilarity}]\label{mainresultcor}
Let $(X_1,\ldots, X_d) \in M_n(\mathbb{C})^d.$
	$$\hat{\rho}(X_1,\ldots,X_d)= \inf_{S\in GL_n(\mathbb{C})} \left\|\bbm SX_1S^{-1} & \ldots & SX_dS^{-1} \ebm\right\|_2.$$
\end{theorem}
Theorem \ref{mainresultcor} follows directly from Theorem \ref{mainresult}.\black
%\red should we actually prove this \black 

We also note the following observation, which follows directly from the Rota-Strang theory, or perhaps even the classical Jordan decomposition.
\begin{theorem}
Let $(X_1,\ldots, X_d) \in M_n(\mathbb{C})^d.$
	$$\hat{\rho}(X_1,\ldots,X_d)=\inf_{N \in \mathfrak{N}}  \|\sum \overline{X_i} \otimes X_i\|_N.$$
\end{theorem}

A tuple $(X_1,\ldots, X_d) \in M_n(\mathbb{C})^d$ is called a \dfn{row contraction} if the block matrix
$\left[ \begin{matrix} X_1 & \ldots & X_d \end{matrix}\right]$ has $2$-norm strictly less than $1.$
Row contractions have been extensively studied for their dilation theoretic properties. (e.g \cite{frazho84,po89,po91,hkms09})
Row contractions are known to satify the following inequality\cite{po89}:
	$$\sup_{
\sum |a_{i_1,\ldots,i_k}|^2 = 1}
\left\|\sum_{1 \leq i_1, \ldots, i_k \leq d}  a_{i_1,\ldots,i_k}X_{i_1}\ldots X_{i_k} \right\|<1.$$
The inequality above gives a hint of the connection of row contractions and outer spectral radius, given the outer spectral radius formula
in Theorem \ref{outerspectralradiusformula}.

The outer spectral radius is characterized in terms of a Lyapunov type condition and equivalence with row contractions.
\begin{theorem}[Popescu \cite{PopescuSimilarity}]\label{mainresult}
	Let $(X_1,\ldots, X_d) \in M_n(\mathbb{C})^d.$
	The following are equivalent:
	\begin{enumerate}
	\item $\hat{\rho}(X_1,\ldots , X_d)<1.$
	\item There is a positive definite matrix $L$ such that
	$L - \sum X_iLX_i^*$ is positive definite.
	\item There is $S\in GL_n(\mathbb{C})$ such that $(SX_1S^{-1}, \ldots, SX_dS^{-1})$ is a row contraction.
	\end{enumerate}
	%There is $S\in GL_n(\mathbb{C})$ such that $(SX_1S^{-1}, \ldots, SX_dS^{-1})$ is a row contraction
	%if and only if $\hat{\rho}(X_1,\ldots , X_d)<1.$ 
\end{theorem}
Theorem \ref{mainresult} is proved in several parts.
The implication  $(1 \Rightarrow 2)$ is part 3 of Proposition  \ref{elempickprop} which explicitly constructs a canonical rationally computable choice for $L$ for which $L - \sum X_iLX_i^*=I$.
The implication $(3 \Rightarrow 1)$  follows from Proposition \ref{leftdirection}.
The equivalence $(2 \Rightarrow 3)$ follows by letting $S = L^{-1/2}$ and $(3 \Rightarrow 2)$ by setting $L = (S^*S)^{-1}.$

From the point of view of free analysis or noncommutative function theory, Theorem \ref{mainresult} is interesting because it shows that any map defined on the set of row contractions extends to all
outer spectral contractions, which includes other domains of interest such as the column contractions. (The author was
originally concerned with these kind of considerations.)
%We note that, in the case $d=1,$ our results essentially follow from the Jordan canonical form.

\section{Preliminaries}
We fix the notation that $I$ is the identity matrix in $M_n(\mathbb{C})$ and $1$ is the identity matrix in $M_{n^2}(\mathbb{C}).$

We will use the following matrix ordering:
for two self-adjoint matrices $A, B \in M_m(\mathbb{C}),$ we say that $A \leq B$ if $B-A$ is positive semidefinite and
we say that $A < B$ if $B-A$ is positive definite. Notably, the notation $A > 0$ means that $A$ is positive definite
and $A \geq 0$ means that $A$ is positive semidefinite.

We adopt a ``vertical tensor notation" to conserve space and enhance visual symmetry during calculations:
	$$\tensor{A}{B} = \flattensor{A}{B}.$$
To proceed we will need to define and describe two important maps: the $\psi$ involution and the partial trace $Q$.

\subsection{The $\psi$ involution}

The \dfn{$\psi$ involution} is the linear map $\psi: M_{n^2}(\mathbb{C}) \rightarrow M_{n^2}(\mathbb{C})$
defined by the relations
	$$\psi(E_{i+(j-1)n,k+(l-1)n}) = E_{i+(k-1)n, j+(l-1)n}$$
where $E_{a,b}$ is the matrix with a $1$ at the $(a,b)$ entry and $0$ everywhere else. Informally,
the $\psi$ involution swaps the positions of $j$ and $k.$
Since $E_{a,b}$ form a basis for  $M_{n^2}(\mathbb{C})$ we can extend the map $\psi$
by linearity to all of $M_{n^2}(\mathbb{C}).$
We formally adopt the notation $$A^\psi = \psi(A).$$
The \dfn{vectorization map} is the linear map $\vecc:M_{n}(\mathbb{C}) \rightarrow \mathbb{C}^{n^2}$
defined by the relations
	$$\vecc E_{i,j} = e_{i+(j-1)n}$$
where $e_k$ is the $k$-th elementary basis vector of $\mathbb{C}^{n^2}.$ We extend $\vecc$ to all
of $M_{n}(\mathbb{C})$ by linearity. The vectorization map is used throughout matrix theory, see \cite{horjoh85}.

The $\psi$ involution will be very useful to us in our analysis of the outer spectral radius and row contractions.
It is especially useful given the following proposition, which shows that the $\psi$ involution has a rich algebraic structure. The $\psi$ involution was used to develop algorithms for understanding finite dimensional matrix algebras \cite{PascoeAlgDim}.
\begin{proposition}\label{psiproperties}
The map $\psi$ satisfies the following properties for any $A, B, C, D \in M_n(\mathbb{C})$
and $E, F \in M_{n^2}(\mathbb{C})$:
\begin{description}
	\item[The map $\psi$ is an involution] $$[E^\psi]^{\psi} = E.$$
	\item[The map $\psi$ takes twisted tensor products to outer products]
	$$\left(\tensor{\overline{A}}{B}\right)^{\psi} = (\vecc B)(\vecc {A})^*,$$
	\item[4-modularity] $$\tensor{\overline{A}}{B}E^{\psi}\tensor{{C}}{D^*} = 
	\left[\tensor{\overline{D}}{B}E\tensor{{C}}{A^*}\right]^{\psi},$$
	\item[Schur Product Property] $$E \geq 0, F\geq 0 \Rightarrow [E^\psi F^\psi]^\psi \geq 0.$$ (Here $G \geq 0$ means that
	$G$ is positive semidefinite.)
\end{description}
\end{proposition}
The proof of Proposition \ref{psiproperties} is elementary and left to the reader.
(For the first three items, it is enough
to check the identity on elementary matrices.
For the Schur product property, it is enough to check the inequality when $E$ and $F$ are rank $1.$)

We note that the Schur Product Property is particularly interesting since it provides a multiplication on
$M_{n^2}(\mathbb{C})$ which preserves positivity, much like the classical Schur product. (The Schur product is also known as
the Hadamard product, or entry-wise multiplication.)

In fact, the Schur Product Property can actually be used to prove the Schur product theorem for the
classical Schur product.
Consider the map
$\tau: M_n(\mathbb{C}) \rightarrow M_{n^2}(\mathbb{C})$
defined by
$$\tau(E_{i,j}) = \tensor{E_{i,j}}{E_{i,j}}.$$
One can show that $A \geq 0$ if and only if $\tau(A) \geq 0,$ and, moreover, one can show that
$$A \circ B = \tau^{-1}\left([\tau(A)^\psi\tau(B)^\psi]^\psi\right)$$
where $A \circ B$ denotes the classical Schur product.

We also note that the properties of taking twisted tensor product to outer products and 4-modularity are equivalent to the definition of the
$\psi$ involution up to multiplication by a constant.

\subsection{The partial trace $Q$}

Another useful map for us will be the \dfn{partial trace $Q$} which is defined to be the map
$Q: M_{n^2}(\mathbb{C}) \rightarrow M_{n}(\mathbb{C})$ which satisfies
	$$Q(E_{i+n(j-1),k+n(l-1)})=\chi_{jl}E_{i,k}.$$
where $\chi_{ab}$ equals $1$ if $a= b$ and $0$ otherwise.

\begin{proposition}\label{Qproperties}
The map $Q$ satisfies the following properties for any $A, B \in M_n(\mathbb{C})$
and $E, F \in M_{n^2}(\mathbb{C})$:
\begin{description}
	\item[The $\psi$-identity identity]
		$$Q([1]^\psi) = I.$$
	\item[The vectorization identity]
		$$Q([\vecc A][\vecc {B}]^*) = AB^*.$$
	\item[Modularity] $$Q\left(\tensor{I}{A}E\tensor{I}{B*}\right) = 
	AQ(E)B^*.$$
	\item[The map $Q$ is positive] $$E \geq 0 \Rightarrow Q(E) \geq 0.$$
	\item[Product property] $$Q(E) \geq 0, F \geq 0 \Rightarrow Q([F^\psi E^\psi]^\psi)\geq 0.$$
\end{description}
\end{proposition}
We leave the above properties as an exercise to the reader.

\section{The outer spectral radius formula}\label{outerspectralradiusformulaproof}

Now we show the outer spectral radius formula given as Theorem \ref{outerspectralradiusformula}.
\begin{proof}[Proof of Theorem \ref{outerspectralradiusformula}]
Recall $$\hat{\rho}(X_1,\ldots , X_d) = \sqrt{\rho\left(\sum \tensor{\overline{X_i}}{X_i}\right)}.$$
Let $W = \sum \tensor{\overline{X_i}}{X_i}.$
By the classical spectral radius formula, it is enough to calculate
	$$\sqrt{\lim_{k\rightarrow \infty} \|W^k\|_F^{1/k}}$$
	where $\|\cdot\|_F$ denotes the Frobenius norm.
Note that $$W^k = \sum_{1 \leq i_1, \ldots, i_k \leq d}
\tensor{\overline{X_{i_1}\ldots X_{i_k}}}{X_{i_1}\ldots X_{i_k}}.$$
So, applying Proposition \ref{psiproperties},
\begin{align*}
	(W^k)^\psi & = \sum_{1 \leq i_1, \ldots, i_k \leq d}
(\text{vec }X_{i_1}\ldots X_{i_k})(\text{vec }{X_{i_1}\ldots X_{i_k}})^*.
\end{align*}
Let $V_k$ be the $n^2$ by $d^k$ with columns $\text{vec }X_{i_1}\ldots X_{i_k}.$
Observe that
	\begin{align*}\|V_k\|_2 &= \sup_{
\sum |a_{i_1,\ldots,i_k}|^2 = 1}
\left\|\sum_{1 \leq i_1, \ldots, i_k \leq d}  a_{i_1,\ldots,i_k}\vecc X_{i_1}\ldots X_{i_k} \right\|, \\
& = \sup_{
\sum |a_{i_1,\ldots,i_k}|^2 = 1}
\left\|\sum_{1 \leq i_1, \ldots, i_k \leq d}  a_{i_1,\ldots,i_k}X_{i_1}\ldots X_{i_k} \right\|_F .
\end{align*}
Note $V_kV_k^* = (W^k)^\psi.$
So $$\|(W^k)^\psi\|_2 = \sup_{
\sum |a_{i_1,\ldots,i_k}|^2 = 1}
\left\|\sum_{1 \leq i_1, \ldots, i_k \leq d}  a_{i_1,\ldots,i_k}X_{i_1}\ldots X_{i_k} \right\|_F^2 .
$$
Observe $\|(W^k)^\psi\|_F = \|W^k\|_F$ and, so,
	$$\|(W^k)^\psi\|_2 \leq \|W^k\|_F \leq n\|(W^k)^\psi\|_2.$$
Namely, we see that
\begin{align*}
	\lim_{k\rightarrow \infty}\|W^k\|_F^{1/2k}& = \lim_{k\rightarrow \infty}\|(W^k)^\psi\|_2^{1/2k} \\
	&= \lim_{k\rightarrow \infty}\sup_{
\sum |a_{i_1,\ldots,i_k}|^2 = 1}
\left\|\sum_{1 \leq i_1, \ldots, i_k \leq d}  a_{i_1,\ldots,i_k}X_{i_1}\ldots X_{i_k} \right\|_F^{1/k}.
\end{align*}

\end{proof}

%Since it is an elementary fact\cite{po91} about row contractions that $$\sup_{
%\sum |a_{i_1,\ldots,i_k}|^2 = 1}
%\left\|\sum_{1 \leq i_1, \ldots, i_k \leq d}  a_{i_1,\ldots,i_k}X_{i_1}\ldots X_{i_k} \right\|_2
%\leq \left\|\bbm X_1 & \ldots &  X_d \ebm\right\|_2^k,$$
%we immediately obtain the following corollary. %\clarify

We now prove a proposition which immediately implies $(3 \Rightarrow 1)$ from Theorem \ref{mainresult} when combined with the
observation that for any $S \in GL_n(\mathbb{C}),$ $$\hat{\rho}(X_1,\ldots , X_d) = \hat{\rho}(SX_1S^{-1},\ldots , SX_dS^{-1}).$$
\begin{proposition}\label{leftdirection}
	If $(X_1,\ldots, X_d) \in M_n(\mathbb{C})^d$ is a row contraction, then
	$\hat{\rho}(X_1,\ldots , X_d)<1.$
\end{proposition}
\begin{proof}
	Let $W$ be as in the proof of Theorem \ref{outerspectralradiusformula}.
	Since $(X_1,\ldots, X_d) \in M_n(\mathbb{C})^d$ is a row contraction, $\sum X_iX_i^* < 1.$
	So, $\sum X_iX_i^* \leq t$ for some $0< t <1.$
	Note $Q(W^\psi) =\sum X_iX_i^*.$
	So $Q([t-W]^\psi) \geq 0$, so, since $[W^k]^\psi\geq 0$, by the product property of $Q$,
		$$Q([tW^k-W^{k+1}]^\psi) = Q([W^k(t-W)]^\psi)\geq 0.$$
	That is,
		$$Q([W^{k+1}]^\psi) \leq tQ([W^k]^\psi).$$
	Inductively, $Q([W^{k}]^\psi) \leq t^kQ(1^\psi) = t^k.$
	Let $\text{vec }u_1(X_1,\ldots,X_d)$ be an eigenvector corresponding to the maximum eigenvalue of $(W^k)^\psi$
	normalized such that
	$$u_1(X_1,\ldots,X_d)=  \sum_{1 \leq i_1, \ldots, i_k \leq d}  a_{i_1,\ldots,i_k}X_{i_1}\ldots X_{i_k}$$ where
	$\sum_{1 \leq i_1, \ldots, i_k \leq d}  |a_{i_1,\ldots,i_k}|^2 = 1$.
	Note that
	$$\left\|\sum_{1 \leq i_1, \ldots, i_k \leq d}  a_{i_1,\ldots,i_k}X_{i_1}\ldots X_{i_k} \right\|_F$$
	is maximized subject to the constraint $\sum_{1 \leq i_1, \ldots, i_k \leq d}  |a_{i_1,\ldots,i_k}|^2 = 1$ by $u_1.$
	Note,
	$$(\text{vec }u_1(X_1,\ldots,X_d))(\text{vec }u_1(X_1,\ldots,X_d))^* \leq (W^k)^\psi.$$ 
	Applying $Q,$ using the vectorization identity and the positivity of $Q$ from Proposition \ref{Qproperties}, we get that
		\begin{align*}u_1(X_1,\ldots,X_d)u_1(X_1,\ldots,X_d)^* & \leq Q([W^k]^\psi) \\& \leq t^k.\end{align*}

	Now, we get that \begin{align*}
		\hat{\rho}(X_1,\ldots, X_d) & = \lim_{k\rightarrow \infty} \sup_{
\sum |a_{i_1,\ldots,i_k}|^2 = 1}
\left\|\sum_{1 \leq i_1, \ldots, i_k \leq d}  a_{i_1,\ldots,i_k}X_{i_1}\ldots X_{i_k} \right\|^{1/k}  \\&\leq
\lim_{k\rightarrow \infty} (t^k)^{1/k} \\ &=  t \\& <1.\end{align*}
	So, by Theorem \ref{outerspectralradiusformula} we are done.
\end{proof}

\section{Construction of the Lyapunov matrix}
 
We will now begin work towards showing the implication $(1 \Rightarrow 2)$ Theorem \ref{mainresult}, which we will accomplish
through an algebraic construction of the matrix $L.$
 
Let $(X_1,\ldots, X_d) \in M_n(\mathbb{C})^d$ such that $\hat{\rho}(X_1,\ldots , X_d)<1.$
We define the \dfn{elementary Pick matrix} corresponding to $(X_1,\ldots, X_d)$ to be 
	$$P=\left[\left(1- \sum \tensor{\overline{X_i}}{X_i}\right)^{-1}\right]^{\psi}.$$
In \cite{PascoeAlgDim}, it was shown that the rank of $P$ gives the dimension of the algebra generated by the $X_i$ and, in fact, that the columns of $P$ span that algebra.
We establish some basic facts about the elementary Pick matrix which culminate in item (3)
which gives implication  $(1 \Rightarrow 2)$ in Theorem \ref{mainresult}.
\begin{proposition}\label{elempickprop}
	Let $(X_1,\ldots, X_d) \in M_n(\mathbb{C})^d$ such that $\hat{\rho}(X_1,\ldots , X_d)<1$ and $P$ be the corresponding
	elementary Pick matrix. The following are true:
	\begin{enumerate}
		\item $P \geq 0,$
		\item $P - \sum \tensor{I}{X_i} P \tensor{I}{X_i^*} = 1^\psi \geq 0,$
		\item Let $L = Q(P).$ $L >0$ and
		$$L - \sum X_iLX_i^* = I > 0.$$ 
	\end{enumerate}
\end{proposition}
\begin{proof}
{\bf(1)} Let
	$$S= \left(\sum \tensor{\overline{X_i}}{X_i}\right)^\psi.$$
We first note that
	$$S = \sum (\vecc X_i)(\vecc X_i)^*$$
since $\psi$ has the property that it takes twisted tensor products to outer products as was established in Proposition \ref{psiproperties}.
Namely, $S \geq 0$ since it is a sum of positive semidefinite rank one matrices.
Note that, since $\hat{\rho}(X_1,\ldots , X_d)<1$ implies that the spectral radius of $S^\psi$ is less than one, we can expand
$P$ using the geometric series:
	\begin{align*}
		P &=\left[\left(1- \sum \tensor{\overline{X_i}}{X_i}\right)^{-1}\right]^{\psi} \\
		& = \left[\left(1- S^\psi\right)^{-1}\right]^{\psi}\\
		&= \left[\sum^\infty_{k=0} (S^\psi)^k\right]^\psi. \\
		& = \sum^\infty_{k=0} [(S^\psi)^k]^\psi
	\end{align*}
	We note that each term $[(S^\psi)^k]^\psi \geq 0$ by the Schur Product Property in Proposition \ref{psiproperties},
so we get that
	$$P = \sum^\infty_{k=0} [(S^\psi)^k]^\psi \geq 0$$
and we are done.

{\bf(2)} Now we want to show that
$$P - \sum \tensor{I}{X_i} P \tensor{I}{X_i^*} \geq 0.$$
First we note that
\begin{align*}
	P - \sum \tensor{I}{X_i}P\tensor{I}{X_i^*}&
	= P -
	\sum \tensor{I}{X_i}\left[\left(1- \sum \tensor{\overline{X_i}}{X_i}\right)^{-1}\right]^{\psi}\tensor{I}{X_i^{*}}
	\\
	& = P - \left[\sum \tensor{\overline{X_i}}{X_i}\left(1- \sum \tensor{\overline{X_i}}{X_i}\right)^{-1}\right]^{\psi}
\end{align*}
by the 4-modularity of $\psi$ from Proposition \ref{psiproperties}.
Now we see that
\begin{align*}
	P - \sum \tensor{I}{X_i}P\tensor{I}{X_i^*}
	& = P - \left[\sum \tensor{\overline{X_i}}{X_i}\left(1- \sum \tensor{\overline{X_i}}{X_i}\right)^{-1}\right]^{\psi} \\
	& = \left[\left(1- \sum \tensor{\overline{X_i}}{X_i}\right)^{-1}\right]^{\psi} - \left[\sum \tensor{\overline{X_i}}{X_i}\left(1- \sum \tensor{\overline{X_i}}{X_i}\right)^{-1}\right]^{\psi} \\
	& =
	\left[\left(1-\sum\tensor{\overline{X_i}}{X_i}\right)
	\left(1- \sum \tensor{\overline{X_i}}{X_i}\right)^{-1}\right]^{\psi}
	\\
	& = [1]^\psi
	\\
	& = (\text{vec }I)(\text{vec }{I})^*
	\\
	& \geq 0
\end{align*}
so we are done.

{\bf(3)} Let $L = Q(P).$ We want to show that $L >0$ and
		$$L - \sum X_iLX_i^* > 0.$$ 
Firstly, we note that
$$L - \sum X_iLX_i^*  = Q\left(P - \sum \tensor{I}{X_i}P\tensor{I}{X_i^*}\right),$$
via the modularity of $Q$ in Proposition \ref{Qproperties}.
Now,
\begin{align*}
L - \sum X_iLX_i^* & = Q\left(P - \sum \tensor{I}{X_i}P\tensor{I}{X_i^*}\right), \\
		& = Q([1]^\psi), \\
		& = I, \\
		& > 0.
\end{align*}
Since $L = Q(P)$ and $P \geq 0,$ we see that $L \geq 0$ by the positivity of $Q$ in Proposition \ref{Qproperties}.
Now, $L \geq L - \sum X_iLX_i^* > 0,$ so we are done.
\end{proof}

\section{The outer spectrum and dynamics of completely positive maps}
We now begin an endeavor to understand the spectral theory of
$$T = \sum \flattensor{\overline{X_i}}{X_i}.$$
The spectrum of $T$ can the though of as the \dfn{outer spectrum} of the tuple $(X_1, \ldots, X_d).$
Our analysis of the outer spectrum gives insight into the dynamics of completely positive maps as was described in the introduction.

\subsection{Maximum eigenvalue is nonnegative} \label{maxsect}
First, we can quickly use our formula for the Lyapunov matrix from Proposition \ref{elempickprop} to show that the spectral radius of $T$ is equal to the maximum eigenvalue of
$T.$
\begin{proposition}
     Let $X=(X_1,\ldots, X_d) \in M_n(\mathbb{C})^d.$
    Let $T = \sum \flattensor{\overline{X_i}}{X_i}.$
    There is a real nonnegative eigenvalue $\lambda$ of $T$ such that $\lambda = \rho(T).$
\end{proposition}
\begin{proof}
    Without loss of generality, assume $\rho(T)=1.$
    Just suppose $1$ is not an eigenvalue of $T.$
    Define
        $$P_t = \left[\left(1- \sum t^2\tensor{\overline{X_i}}{X_i}\right)^{-1}\right]^{\psi}.$$
    Clearly, $P_t$ is well-defined, continuous, and finite for $t \in [0,1].$
    Let $L_t = Q(P_t).$
    For each $0\leq t < 1,$ we see that
        $L_t - t^2\sum X_iL_tX_i^* = I,$
    by Proposition \ref{elempickprop} because $\rho(tX)<1.$
    So, by continuity,
        $$L_1 - \sum X_iL_1X_i^*=I>0.$$
    Therefore, by Theorem \ref{mainresult}, we see that $\hat{\rho}(X)=\rho(T) < 1,$ which is a contradiction.
\end{proof}

To prove that the maximual spectrum actually contains a nonnegative element is somewhat more subtle. We now prove Theorem \ref{maximalreal}.
\begin{proof}
If $T$ is nilpotent, there is nothing to prove. Suppose $\rho(T) =1.$
Recall $\Lambda$ is the closure of the subgroup of the torus $\mathbb{T}^{m_T}$ generated by the point $\tau=(\lambda_1, \ldots, \lambda_{m_T})$
where the $\lambda_i$ are the elements of the maximal spectrum, where $m_T$ is the cardinality of the maximal  spectrum of $T.$
Recall there was a bijective map $T_\lambda$ from $\Lambda$ to the limit points of the sequence $\frac{T^n}{\|T^n\|}.$
Note that 
    $T^\psi_\lambda \geq 0$
by the Schur product property for the $\psi$ involution.
Moreover, $T_\lambda^\psi \neq 0.$
Now, recall,
    $$\hat{T} = \int_{\Lambda} T_\lambda \dd \lambda =\lim_{N\rightarrow \infty} \frac{1}{N}\sum^N_{n=1} \frac{T^n}{\|T^n\|}.$$
So, $\hat{T}^\psi\geq 0$ and $\hat{T}^\psi \neq 0.$
By considering the Jordan decomposition of $T$, we see that any generalized eigenvector $v$ of $T$
corresponding to an eigenvalue which is not $1$ must be sent the zero vector by $\hat{T}$ by looking that the formula for $\hat{T}.$
Moreover, if $1$ is an eigenvalue, but does not have maximal degeneracy index, $\hat{T}$ would also send any generalized eigenvector $v$ with eigenvalue $1$ to $0.$
Therefore, since $\hat{T}\neq 0,$ $1$ must be an eigenvalue of $T$ with maximal degeneracy index.
\end{proof}

\subsection{Structure of the $T_\lambda$}
We now prove Theorem \ref{dyntheorem}.
\begin{proposition}\label{dynprop}
    Let $T = \sum \flattensor{\overline{X_i}}{X_i}.$
    Additionally, assume that $\hat{\rho}(X_1,\ldots , X_d)=1.$
    \begin{enumerate}
    \item
    There is an $a \in \mathbb{N}$ such that  for every $\lambda \in \Lambda$
    there are $A_{\lambda,1}\ldots A_{\lambda,a}\in M_n(\mathbb{C})$ such that
    $T_\lambda = \sum^a_{j=1} \flattensor{\overline{A_{\lambda,j}}}{A_{\lambda,j}}.$
    \item
     There are $B_{k}\in M_n(\mathbb{C})$ whose span is a nonzero ideal in the algebra generated by $X_1,\ldots, X_d$ such that
    $\hat{T} = \sum_k \flattensor{\overline{B_{k}}}{B_{k}}.$ Moreover, the $A_{\lambda,j}$ are contained in the span of the $B_{k}$'s.
    \end{enumerate}
\end{proposition}
\begin{proof}
    Note $T_\lambda^\psi \geq 0.$ So $$T_\lambda^{\psi} = \sum \vecc{A_{\lambda,j}}(\vecc{A_{\lambda,j}})^*.$$
    So, since $\psi$ interchanges Kronecker products and outer products,
        $$T_\lambda = \sum^a_{j=1} \flattensor{\overline{A_{\lambda,j}}}{A_{\lambda,j}}.$$
    Moreover, since
        $$\hat{T}^\psi = \int_{\Lambda} T_\lambda^{\psi} \dd \lambda,$$
    we see that $\hat{T}$ has the appropriate form and $\vecc{A_{\lambda,j}}$ is in its range for each choice of $\lambda$ and $j.$ (That is, something is in the kernel of $\hat{T}^\psi$ if and only if it is in the kernel of every $T_\lambda^\psi.$)
    Now, it remains to be seen that the $B_k$ span an ideal. 
    Recall that
        $\hat{T}T = T\hat{T} = \hat{T}.$
    That is,
       $$\sum_i \sum_k \flattensor{\overline{X_iB_{k}}}{X_iB_{k}} = \sum_i \sum_k \flattensor{\overline{B_{k}X_i}}{B_{k}X_i}=\sum_k \flattensor{\overline{B_{k}}}{B_{k}}.$$
    By taking the $\psi$ map of the relation, we see that 
         $$\sum_i \sum_k \vecc X_iB_{k}(\vecc X_iB_{k})^* = \sum_i \sum_k \vecc B_{k}X_i(\vecc B_{k}X_i)^*=\sum_k \vecc B_{k}(\vecc B_{k})^*,$$
    which implies that each $X_iB_{k}$ and $B_{k}X_i$ is in the span of the $B_k$'s and therefore they span an ideal in the algebra generated by the $X_i$'s.
\end{proof}

We now need the following lemma.
\begin{lemma}\label{poslem}
Let $U= \sum\flattensor{\overline{A_{k}}}{A_{k}},$ $V = \sum \flattensor{\overline{B_{k}}}{B_{k}}.$
If the $B_k$'s are contained in the span of the $A_k$'s then there is an $\varepsilon$ such that $\phi_U(H) \geq \varepsilon\phi_V(H).$ 
\end{lemma}
\begin{proof}
    Note there is an $\varepsilon$ such that $$U^\psi \geq \varepsilon V^\psi.$$ So $(U-\varepsilon V)^\psi$ is positive semi-definite. Therefore, $U-\varepsilon V= \flattensor{\overline{C_{k}}}{C_{k}}$
    and so $\phi_{U-\varepsilon V}$ is a completely positive map.
\end{proof}
We now prove Theorem \ref{fullmaximal}.
\begin{proposition}\label{fullmaximalprop}
     Let $T = \sum \flattensor{\overline{X_i}}{X_i}.$
    Additionally, assume that $\hat{\rho}(X_1,\ldots , X_d)=1.$
    If $X_1, \ldots, X_d$ generate the full algebra of $n$ by $n$ matrices, then $\hat{T}$ has rank $1.$
    Moreover, in such a case, $\hat{T}=\sum_k \flattensor{\overline{B_{k}}}{B_{k}}$ where the $B_k$'s span $M_n(\mathbb{C}).$    
    That is, the ideal described in Theorem \ref{dyntheorem} is equal to the whole algebra, as the only nonzero ideal in $M_n(\mathbb{C})$ is  $M_n(\mathbb{C})$  itself.
\end{proposition}
\begin{proof}
    First note that $\hat{T}=\sum_k \flattensor{\overline{B_{k}}}{B_{k}}$ where the $B_k$'s span $M_n(\mathbb{C})$ by Theorem \ref{dyntheorem} and the fact that $M_n(\mathbb{C})$ is simple.
    Note that is such a case, $\phi_{\hat{T}}(H) \geq \varepsilon (\Tr H) I$ by Lemma \ref{poslem}, since $(\Tr H) I = \sum E_{ij}HE_{ij}^*.$ 
    We will work to establish that there is a positive semidefinite matrix $H$ such that $\phi_{\hat{T}}(H)$ has a kernel.
    If $\hat{T}^2 = 0,$ for any postive matrix $H,$ either $H$ is in the kernel or $\phi_{\hat{T}}(H)$ is.
    
    Now, suppose $\hat{T}^2 = \hat{T}.$
    Just suppose the rank of $\hat{T}$ is greater than $1.$
    Now there must be $v_1, v_2 \in \mathbb{C}^n$ such that $T \overline{v_j}\otimes v_j \neq 0,$ and 
    $T \overline{v_1}\otimes v_1$ and $T \overline{v_2}\otimes v_2$ are linearly indepenent. Let $W_1$ and $W_2$ be matrices such that $\vecc W_j = T \overline{v_j}\otimes v_j.$
    Note $W_j = \sum_i X_iv_j(X_iv_j)^*.$ Note $\phi_{\hat{T}}(W_j) = W_j.$
    Since $W_1$ and $W_2$ are linearly independent positive semidefinite matrices, there is $W_3$ in their span which is positive semidefinite and singular.
    Moreover, $\phi_{\hat{T}}(W_3) = W_3$ which is singular. This contradicts the fact that $\phi_{\hat{T}}(H) \geq \varepsilon (\Tr H) I.$
\end{proof}

By considering the form of $\hat{T}=\vecc V (\vecc W)^*$ from the above theorem, and the fact that $\hat{T}^{\psi}=\overline{W}\otimes V$ must be positive, we see that both the $V, W$ must be positive and, therefore we can take square roots. An elementary calulation gives that the appropriate conjugation by $V^{1/2}$ and $W^{1/2}$ transform the corresponding completely positive map into either a unital or trace preserving map respectively.
\begin{corollary}\label{fullmaximalcor}
     Let $T = \sum \flattensor{\overline{X_i}}{X_i}.$
    Additionally, assume that $\hat{\rho}(X_1,\ldots , X_d)=1,$ and $X_1, \ldots, X_d$ generate the full algebra of $n$ by $n$ matrices. By Theorem \ref{fullmaximal},
	$$\hat{T}=\vecc V (\vecc W)^*.$$
	Then, $$\acute{X}= (V^{-1/2}X_1V^{1/2},\ldots, V^{-1/2}X_dV^{1/2})$$ is a row co-isometry and
	$$\grave{X}= (W^{1/2}X_1W^{-1/2},\ldots, W^{1/2}X_dW^{-1/2})$$ is a column isometry.
	Moreover, the corresponding $\phi_{\acute{T}}$ is unital and $\phi_{\grave{T}}$ is trace preserving.
\end{corollary}

\subsection{Spectral structure of $\sum X_i^{\otimes 2k}$} \label{ParilloSection}
	Blondel and Nesterov\cite{blondnest05}
	commented that the eigenvalues of the matrix $\sum X_i^{\otimes 2k}$ appeared to be structured.
	In Parrilo and Jadbabaie\cite{par08}, it was shown that one can project $\sum X_i^{\otimes 2k}$ onto some special invariant subspace
	which corresponds to the ``symmetric algebra of a vector space." We will now briefly explain these phenomena.
	
	We will restrict our attention to matrices over the reals.
	Let $(X_1, \ldots, X_d)\in M_n(\mathbb{R}).$
	Let $Y_{k} =\sum X_i^{\otimes k}.$
	For any permutation $\sigma: \{1, \ldots, k\} \rightarrow \{1, \ldots, k\},$ we define
	$P_\sigma$ to be the map which takes
$v_1\otimes v_2 \otimes \ldots \otimes v_k$ to $v_{\sigma(1)}\otimes v_{\sigma(2)} \otimes \ldots \otimes v_{\sigma(k)}.$
	A direct calculation gives the following.
	\begin{rem}
		$P_\sigma Y_k = Y_kP_\sigma.$ Thus, if $v$ is an eigenvector for $Y_k$ with eigenvalue $\lambda,$  then so is $P_{\sigma}v.$ 
	\end{rem}
	We also note that one can derive the structure of $P_\sigma$ (and thus the invariant subspaces of $Y_k$) using
	classical representation theory of the symmetric group, although this is rather involved
	\cite{sagan}.
	
	Moreover, in test cases, we saw that the eigenvalue with maximum modulus of $Y_{2k}$ was always real and positive, and, in nontrivial examples,
occured with multiplicity $1,$ which reflects the general situation described  in Theorem \ref{maximalreal}. In such a case, the eigenvector corresponding to the maximum eigenvalue must satisfy $P_{\sigma}v = \pm v.$
(In test cases, we found that $P_{\sigma}v = v.$)
In general, the space $V = \{v | P_\sigma v = v \forall \sigma \}$ is an invariant subspace of $Y_k.$ In fact, it follows from Parrilo and Jadbabaie\cite[Theorem 4.2]{par08} that
		$$\rho(X_1,\ldots,X_d) \leq \rho(Y_{2k}|_V)^{1/2k} \leq \rho(Y_{2k})^{1/2k}.$$
		
	We now remark that one can rephrase the above in terms of a certain action on polynomials.
	We the naturals action $\tau$ of each $X_i$ on a ring of polynomials $\mathbb{R}[e_1,\ldots, e_n]$
	given by $\tau_{X_i} \cdot p(e) = p(X_ie)$ (here $e$ is the column vector with entries $e_i$.)
One can check that the action of $Y_k$ on $V$ is isomorphic to the action $\sum \tau_{X_i}$ on homogeneous polynomials
of degree $k$ in $\mathbb{R}[e_1,\ldots, e_n],$ denoted $\mathbb{R}[e_1,\ldots, e_n]_k.$ That is, we have the following theorem which follows immediately from the Parrilo and Jadbabaie inequality above and Corollary \ref{osrfccor}.
%where the inner product is defined so that
%monomials are orthogonal and $$\|e_1^{k_1}\ldots e_n^{k_n}\| = \sqrt{\frac{k!}{k_1!\ldots k_n!}}.$$
	
	\begin{theorem} \label{parrilorephrase}
		 Let $(X_1,\ldots, X_d) \in M_n(\mathbb{R})^d.$ Let $\tau_{X_i}$ denote the natural action of
		 $X_i$ on  polynomials $\mathbb{R}[e_1,\ldots,e_n].$
		$$\rho(X_1,\ldots, X_d) = \lim_{k\rightarrow \infty} \rho\left(\sum \tau_{X_i}|_{\mathbb{R}[e_1,\ldots,e_n]_{2k}} \right)^{1/2k}.$$
	\end{theorem}
	We note that the total dimension of
of the homogenous polynomials of degree $k$ in $n$ variables is equal to $h_{n,k} ={{n+k-1}\choose{n-1}} \approx k^{n-1}/n!$, so we can realize
$\tau_{k}$ as an $h_{n,k}$ by $h_{n,k}$ matrix. Theorem \ref{parrilorephrase} is equivalent to the method in Parrilo and Jadbabaie \cite{par08} called $\rho_{SR,2d},$ and gives a natural
interpretation of $Y_k$ restricted to $V.$
%Furthermore, evaluation of the natural action does not require the computation
%of permanents.
%Fast methods of computing this approximation are now known, see \cite{par08}.

%In turn, we see that it is enough to consider the eigenvalues of $X_k$ restricted to the subspace
%$V = \{v | Pv = v\}.$ One can check that the action of $X_k$ on $V$ is isomorphic to the sum of the actions of each $X_i$ on homogeneous polynomials
%of degree $k$ in $\mathbb{C}[e_1,\ldots, e_n]$ given by $X_i \cdot e_jw = (X_ie_j)(X_i \cdot w)$ where the inner product is defined so that
%monomials are orthogonal and $\|e_1^{k_1}\ldots e_n^{k_n}\| = \sqrt{\frac{k}{k_1!\ldots k_n!}}.$ We note that the total dimension of
%of the homogenous polynomials of degree $k$ in $n$ variables is equal to ${{n+k-1}\choose{n-1}} \approx k^{n-1}$.
%\black
%Noting that, generically, the largest eigenvalue is distinct we see that 
%Notably, Corollary \ref{osrfccor} gives a method one can use the outer spectral radius obtain a family of asymptotically tight approximations of joint spectral radius.

%\red
%We say a tuple $(X_1,\ldots, X_d) \in M_n(\mathbb{C})^d$ is \dfn{equivalent} to $(Y_1,\ldots, Y_d) \in M_m(\mathbb{C})^d,$ also written as $(X_1,\ldots, X_d) \sim (Y_1,\ldots, Y_d),$
%if for any free polynomial $p$ in $d$ variables
%	$$p(X_1,\ldots,X_d) = 0 \Leftrightarrow p(Y_1,\ldots,Y_d) = 0.$$
%Essentially, two points are equivalent if they have the same algebraic properties.\black

\bibliography{references}

\begin{thebibliography}{10}

\bibitem{par14}
A.A. Ahmadi, R.~Jungers, P.A. Parrilo, and M.~Roozbehani.
\newblock { Joint Spectral Radius and Path-Complete Graph Lyapunov Functions }.
\newblock {\em SIAM J. Control and Optimization}, 52(1):687--717, 2014.

\bibitem{BlondelBirther}
Vincent~D. Blondel.
\newblock The birth of the joint spectral radius: An interview with gilbert
  strang.
\newblock {\em Linear Algebra and its Applications}, 428(10):2261 -- 2264,
  2008.
\newblock Special Issue on the Joint Spectral Radius: Theory, Methods and
  Applications.

\bibitem{blondnest05}
Vincent~D. Blondel and Yurii Nesterov.
\newblock Computationally efficient approximations of the joint spectral
  radius.
\newblock {\em SIAM Journal on Matrix Analysis and Applications},
  27(1):256--272, 2005.

\bibitem{blondel2}
Vincent~D. Blondel and John~N. Tsitsiklis.
\newblock {The Lyapunov exponent and joint spectral radius of pairs of matrices
  are hard--when not impossible--to compute and to approximate}.
\newblock {\em Mathematics of Control, Signals and Systems}, 10(1):31--40,
  1997.

\bibitem{blondel}
Vincent~D. Blondel and John~N. Tsitsiklis.
\newblock The boundedness of all products of a pair of matrices is undecidable.
\newblock {\em Systems and Control Letters}, 41:2:135–--140, 2000.

\bibitem{cho72}
M.-D. Choi.
\newblock {\em Positive linear maps on {C*-algebras}}.
\newblock PhD thesis, University of Toronto, 1972.

\bibitem{DaubechiesLagarias}
I.~Daubechies and J.~Lagarias.
\newblock Two-scale difference equations. i. existence and global regularity of
  solutions.
\newblock {\em SIAM Journal on Mathematical Analysis}, 22(5):1388--1410, 1991.

\bibitem{EKH}
David~E. Evans and Raphael Høegh-Krohn.
\newblock Spectral properties of positive maps on c*-algebras.
\newblock {\em Journal of the London Mathematical Society}, s2-17(2):345--355,
  1978.

\bibitem{frazho84}
A.E. Frazho.
\newblock Complements to models for noncommuting operators.
\newblock {\em J. Funct. Anal.}, 59(3):445 -- 461, 1984.

\bibitem{hkms09}
J.W. Helton, I.~Klep, S.~McCullough, and N.~Slinglend.
\newblock Noncommutative ball maps.
\newblock {\em J. Funct. Anal.}, 257:47--87, 2009.

\bibitem{horjoh85}
R.A. Horn and C.R. Johnson.
\newblock {\em Matrix Analysis}.
\newblock Cambridge University Press, Cambridge, 1985.

\bibitem{kraus71}
K.~Kraus.
\newblock General state changes in quantum theory.
\newblock {\em Ann. Phys.}, 64:311--335, 1971.

\bibitem{Lagro2017}
Matthew Lagro, Wei-Shih Yang, and Sheng Xiong.
\newblock A perron--frobenius type of theorem for quantum operations.
\newblock {\em Journal of Statistical Physics}, 169(1):38--62, Oct 2017.

\bibitem{Shitov}
M.~{Michałek} and Y.~{Shitov}.
\newblock Quantum version of wielandt’s inequality revisited.
\newblock {\em IEEE Transactions on Information Theory}, pages 1--1, 2019.

\bibitem{Hermes}
Alexander Müller-Hermes, Daniel Stilck~França, and Michael~M. Wolf.
\newblock Entropy production of doubly stochastic quantum channels.
\newblock {\em Journal of Mathematical Physics}, 57(2):022203, 2016.

\bibitem{par08}
P.~A. Parrilo and A.~Jadbabaie.
\newblock {Approximation of the joint spectral radius using sums of squares}.
\newblock {\em Linear algebra and its Applications}, 428(10):2385--2402, 2008.

\bibitem{PascoeAlgDim}
J.E. Pascoe.
\newblock An elementary method to compute the algebra generated by some given
  matrices and its dimension.
\newblock {\em Linear Algebra and its Applications}, 571:132 -- 142, 2019.

\bibitem{po89}
G.~Popescu.
\newblock Isometric dilations for infinite sequences of noncommuting operators.
\newblock {\em Trans. Amer. Math. Soc.}, 316:523--536, 1989.

\bibitem{po91}
G.~Popescu.
\newblock {Von Neumann inequality for $(B({\mathcal{H}})\sp n)\sb 1$}.
\newblock {\em Math. Scand.}, 68:292--304, 1991.

\bibitem{PopescuSimilarity}
G.~Popescu.
\newblock Similarity problems in noncommutative polydomains.
\newblock {\em J. Funct. Anal.}, 267(11):4446--4498, 2014.

\bibitem{Raginsky}
Maxim Raginsky.
\newblock Entropy production rates of bistochastic strictly contractive quantum
  channels on a matrix algebra.
\newblock {\em Journal of Physics A: Mathematical and General},
  35(41):L585--L590, oct 2002.

\bibitem{rotastrang}
G.~C. Rota and G.~Strang.
\newblock A note on the joint spectral radius.
\newblock {\em Indag. Math.}, 22:379--381, 1960.

\bibitem{rotamodel}
Gian-Carlo Rota.
\newblock On models for linear operators.
\newblock {\em Communications on Pure and Applied Mathematics}, 13(3):469--472,
  1960.

\bibitem{sagan}
B.~Sagan.
\newblock {\em The symmetric group. Representations, combinatorial algorithms,
  and symmetric functions, 2nd ed., Graduate Text in Mathematics 203}.
\newblock Springer-Verlag, 2001.

\bibitem{Sanz}
M.~{Sanz}, D.~{Perez-Garcia}, M.~M. {Wolf}, and J.~I. {Cirac}.
\newblock A quantum version of wielandt's inequality.
\newblock {\em IEEE Transactions on Information Theory}, 56(9):4668--4673, Sep.
  2010.

\bibitem{shred}
Robert Schrader.
\newblock Perron-{F}robenius theory for positive maps on trace ideals.
\newblock In {\em Mathematical physics in mathematics and physics ({S}iena,
  2000)}, volume~30 of {\em Fields Inst. Commun.}, pages 361--378. Amer. Math.
  Soc., Providence, RI, 2001.

\bibitem{xu2011}
Jianhong Xu and Mingqing Xiao.
\newblock A characterization of the generalized spectral radius with
  {K}ronecker powers.
\newblock {\em Automatica}, 47(7):1530 -- 1533, 2011.

\end{thebibliography}
\bibliographystyle{plain}

\end{document}